\documentclass[12pt]{article}
\usepackage{amssymb,amsmath}
\usepackage{amsthm}
\usepackage{amscd}
\usepackage{stmaryrd}

\theoremstyle{plain}

\newtheorem{theorem}{Theorem}[section]
\newtheorem{theorem*}[theorem]{Theorem }
\newtheorem{proposition}[theorem]{Proposition}

\newtheorem{corollary}[theorem]{Corollary}

\newtheorem{lemma}[theorem]{Lemma}

\DeclareMathOperator{\End}{End}

\DeclareMathOperator{\Herm}{Herm}

\DeclareMathOperator{\Aut}{Aut}
\DeclareMathOperator{\rank}{rank}
\DeclareMathOperator{\id}{id}
\DeclareMathOperator{\Id}{Id}

\DeclareMathOperator{\tr}{tr}
\DeclareMathOperator{\Tr}{Tr}

\DeclareMathOperator{\Det}{Det}

\begin{document}
\title{Unitarizable vector-valued holomorphic series and the Laplace transform : an example}

\author{Jean-Louis Clerc}

\date{}
\maketitle

\abstract{For $T_\Omega$ a Hermitian symmetric  tube-type domain, a family $(\pi_\mu)_{\mu\in \mathbb C}$ of holomorphic vector-valued representations is studied. The corresponding Wallach set is determined. The main tool is a realization of the representations as weighted $L^2$-spaces on the cone $\Omega$ through the Laplace transform.}
\bigskip

{2020 MSC. Primary 22E46; Secondary 32M15; 44A10 }
\bigskip

\hskip 4cm {\bf Dedicated to Karl Hofmann} 

 \hskip 4cm {\bf on the occasion of his 90th birthday}

\section*{Introduction}

The holomorphic discrete series has been studied intensively after its introduction by Harisch Chandra. In particular the analytic continuation of the series and the possible unitarizable representations beyond the discrete series is a difficult problem, solved in full generality in the 80's (see \cite{ehw}, \cite{j}). Earlier, for scalar-valued cases, the description of the \emph{Wallach set} \cite{w} was an important result. For tube-type domains in particular, the Laplace transform offers explicit and nice realizations of the singular representations (see \cite{rv}, \cite{fk}). There were a few attempts to study the vector-valued cases by a similar approach (see \cite{k}, \cite{c95}, \cite{d}, \cite{hn} ), but far from being conclusive. The present work addresses these questions on a modest but typical example.

The theory of tube-type domains which are Hermitian symmetric spaces is intimately connected to the theory of Euclidean Jordan algebras, as developed in \cite{fk}. To any (say) simple Euclidean Jordan algebra $J$ is associated a symmetric cone $\Omega$, and the tube-domain is $J\oplus i\Omega$. The group $G$ of bi-holomorphic automorphisms has a family of generators which are easily described in the framework of Jordan algebras. The maximal compact subgroup $U$ of $G$ is a real form of the  \emph{structure group} of $\mathbb J$, the complexification of $J$. Hence among the irreducible representations of $U$ one can single out those which are irreducible sub-representations of the natural action of $U$ on the polynomial algebra on $\mathbb J$. 
The simplest component is the space of degree one homogeneous polynomials (i.e. linear forms). This allows to define a family of holomorphic representations $(\pi_\mu), \mu\in \mathbb C$ (in fact projective representations, which can be realized as representations of the universal covering of $G$). For $\mu$ real-valued and large enough, the representations are unitary and realized on \emph{weighted Bergman spaces}. Their reproducing kernels $\mathcal Q_\mu$ are easy to determine and it is a natural question to ask for which values of $\mu$ is $\mathcal Q_\mu$ still positive-definite, or said in different terms to determine the \emph{Wallach set} corresponding to this family of unitarizable vector-valued holomorphic representations. Simple Euclidean Jordan algebras are characterized by three numbers $n,r,d$ which satisfy the relation $n=r+\frac{r(r-1)}{2} d$. The main theorem of the present article can be formulated as follows.\footnote{tacitly assumed is the condition $r\geq 2$, the case where $r=1$ corresponds to the complex half-line and has only scalar representations}

\begin{theorem}\label{Wallachset}
 The kernel $\mathcal Q_{\mu}$ is positive-definite if and only if
$\mu$ belongs to 
\[\left\{2\, \frac{d}{2}\,, 3\,\frac{d}{2}\,,\cdots, (r-1)\frac{d}{2}\right\} \bigcup\ \boldsymbol {\Big [} \frac{rd}{2}, +\infty\Big)\ .
\]
\end{theorem}

In order to investigate other cases, the obstacle is in the complexity of the decomposition of the chosen representation of $L$, when restricted to the subgroup $K$ of automorphisms of the Jordan algebra $J$. In the present example, there are only two components and it is possible to handle the case. The investigation through the Laplace transform gives interesting insights in the problem, and our results gives some motivation for considering more examples. 

\section{Euclidean Jordan algebra and tube-type domains}

Let $J$ a simple Euclidean Jordan algebra. Our main reference on the subject is \cite{fk} and we usually follow their notation. In particular, a simple Euclidean Jordan algebra is characterized by three integers : its dimension $n$, its \emph{rank} $r$ and an integer $d$, which satisfy
\[n= r+ \frac{r(r-1)}{2}\, d\ .
\]

The  standard inner product on $J$ is given by 
\[x,y\in J, \qquad \qquad (x\vert y)=\tr(xy), \]
where $\tr$ is used for the \emph{trace} function on $J$.

Let $L$ be the neutral component of the \emph{structure group} of $J$. Then $L$ is a reductive group, and $K=L\cap O(J)$ is a maximal compact subgroup of $L$ which is also the stabilizer in $L$ of the unit element $e$ and the neutral component of the automorphisms group of $J$.

Denote by $\det$ the \emph{determinant} of the Jordan algebra $J$. Recall that $\det$ is a homogeneous polynomial of degree $r$ on $J$.

Let $\chi:L\longrightarrow \mathbb R^+$ be the character on $L$ which satisfies the following identity, 
\[\det \ell x = \chi(\ell) \det x\ ,
\]
for $\ell\in L$ and any $x\in J$.

Let $P$ be the quadratic representation of $J$, and recall that for any $x\in \Omega$ 
\begin{equation}\label{chiP}
P(x)\in L\quad \text{and} \quad \chi\big(P(x)\big) = (\det x)^2\ .
\end{equation}
Another useful formula, valid for any $\ell\in L$ is
\[\Det \ell = \chi(\ell)^{n/r}\ .
\]

Let $\Omega$ be the positive cone of $J$. The measure 
\[d^* x= (\det x)^ {-\frac{n}{r} }dx\] 
is invariant under the action of $L$.

Introduce the \emph{$\Gamma$-function} of the cone as the integral
\[\Gamma_\Omega(\lambda) =\int_\Omega e^{-\tr x} \det(x)^{\lambda} \, d^*x
\]
The integral converges absolutely for $\Re(\lambda)>(r-1)\frac{d}{2}= \frac{n}{r}-1$ and
\[\Gamma_\Omega(\lambda)= (2\pi)^{\frac{n-r}{2}} \Gamma(\lambda) \Gamma\left(\lambda-\frac{d}{2}\right)\dots \Gamma\left(\lambda-(r-1)\frac{d}{2}\right)\ .
\]

For $\alpha\in \mathbb C, \Re(\alpha)> (r-1)\frac{d}{2}$, recall the \emph{Riesz integral} given for a Schwarz function $\varphi\in \mathcal S(J)$ by
\begin{equation}\label{Riesz}
T_\alpha(\varphi) = \frac{1}{\Gamma_\Omega(\alpha)}\int_\Omega \varphi(x) (\det x)^{\alpha} d^*x\ .
\end{equation}

\begin{proposition} For $\varphi$ in $\mathcal S(J)$, $T_\alpha(\varphi)$ admits an analytic continuation as an entire function of  $\alpha$. Moreover, the analytic continuation $T_\alpha$ is a tempered distribution on $J$ for all $\alpha\in \mathbb C$ and satisfies
\begin{equation}
T_\alpha(\varphi\circ \ell^{-1}) = \chi( \ell)^\alpha T_\alpha(\varphi)\ .
\end{equation}
\end{proposition}
The closure $\overline\Omega$ contains $r+1$ orbits under the action of $L$, namely
\[\overline \Omega= \bigsqcup_{k=0}^r\,\Omega^{(k)}\ ,
\]
where $\Omega^{(k)}$ is the set of elements of rank $k$ in $\overline \Omega$. Notice that $\Omega^{(0)}=\{0\}$ and $\Omega^{(r)}= \Omega$.
\begin{proposition}\label{nuk}
For $\alpha = k\frac{d}{2}, 0\leq k\leq r-1$, the distribution $T_\alpha$ is a positive measure (from now on denoted by $\nu_k$) supported by $\overline {\Omega^{(k)}}$. The measure $\nu_k$ is quasi-invariant by $L$ and satisfies
\begin{equation}\label{dnukcov}
\nu_k(\ell \,.\,) = \chi (\ell)^{k\frac{d}{2}}\,\nu_k(\,.\,)\ .
\end{equation}
\end{proposition}
Let $c\neq 0$ be an idempotent of $J$. Then the \emph{Peirce decomposition} of $J$ with respect to $c$ is given by
\[J=J(c,1)\oplus J(c,1/2)\oplus J(c,0)\ ,
\]
where for $\alpha=0,1/2, 1$, \[J(c,\alpha) = \{ x\in J,\ cx=\alpha x\}\ .\]
A \emph{Jordan frame} of $J$ is a collection $(c_1,c_2,\dots, c_r)$ of strongly orthogonal  primitive idempotents which satisfy
\[e=c_1+c_2+\dots +c_r\ .
\] 
To each Jordan frame corresponds a \emph{Peirce decomposition}
\[J= \bigoplus_{1\leq i\leq j\leq r}  J_{ij},\]
where for $1\leq i\leq r$ and $i+1\leq j\leq r$, 
 \[J_{ii} = J(c_i,1) = \mathbb R c_i, \qquad J_{ij}= J(c_i, 1/2)\cap J(c_j, 1/2)\ .
\]
The spaces $J_{ij}$ for $1\leq i<j\leq r$ have the same dimension $d$.

Let $\mathbb J$ be the complexification of $J$. Extending  the Jordan algebra structure of $J$  to $\mathbb J$ in a $\mathbb C$-linear way, $\mathbb J$ becomes a complex simple  Jordan algebra. Let $\mathbb L$ be the connected component of the structure group of $\mathbb J$, which is a reductive group. Extend the Euclidean inner product of $J$ to a Hilbertian product on $\mathbb J$ and  let $U(\mathbb J)$ be the corresponding unitary group. Then
\begin{equation}\label{defU}
U=\mathbb L\cap U(\mathbb J)
\end{equation}
is a connected maximal compact subgroup of $\mathbb L$. Notice that $L$ and $U$ are two real forms of $\mathbb L$.

Let
\[T_\Omega = \{ z=x+iy\in \mathbb J, y\in \Omega\}\ 
\]
be the corresponding \emph{tube domain}. Let $G$ be the neutral component of the group of biholomorphic diffeomorphisms of $T_\Omega$. The group $G$ is generated by 
\smallskip

$\bullet$ the group $L$, after complex extension to $\mathbb J$ of its elements
\smallskip

$\bullet$ the group of translations $N=\{t_u : z\mapsto z+u, u\in J\}$ 
\smallskip

$\bullet$ $\iota$ the \emph{inversion} $z\longmapsto -z^{-1}\ .$
\smallskip

The group $G$ acts transitively on $T_\Omega$ and the space $T_\Omega$,  equipped with the Bergman metric, is a non compact Hermitian symmetric space. 

Let $g\in G$ and let $z\in T_\Omega$. Then the complex differential $Dg(z)$ is an  element of the complex structure group $\mathbb L$ of the complex Jordan algebra $\mathbb J$. For the generators of $G$, notice that for $z\in T_\Omega$
\smallskip

$\bullet$ $D\ell (z) = \ell$ for $\ell\in L$.
\smallskip

$\bullet$ $ D t_u (z)= \id$ for $u\in J$
\smallskip

$\bullet$ $D\iota(z) = P\left(z^{-1}\right)=P(z)^{-1}$ for the inversion $\iota$.
\smallskip

For $g\in G$ and $z\in T_\Omega$, let $J(g,z) =Dg(z)$ be the differential of the map $g$ at $z$. An important result is that 
\[J(g,z)\  \in\ \mathbb L\ ,\]
as can verified on the generators of $G$ and extended to $G$ by using the chain rule.

Let $\widetilde U$ be the stabilizer of the origin $ie\in T_\Omega$. Then $\widetilde U$ is a maximal compact subgroup of $G$. To describe this subgroup, it is convenient to refer to the \emph{bounded realization} of $T_\Omega$.  Let $\vert \, .\, \vert $ be the \emph{spectral norm} on $\mathbb J$ and let $D$ be the corresponding open unit ball
\[D = \{ w\in \mathbb J,\quad \vert w\vert <1\}\ .
\]
Let $G(D)^0$ be the neutral component of the group of biholomorphic diffeomorphisms of $D$. Then the stabilizer of $0\in D$ in $G(D)^0$ is equal to $U$ as defined by  \eqref{defU}.

Define the \emph{Cayley transform} to be the rational map $c$ on $\mathbb J$ given by
\[c(w) = i(e+w)(e-w)^{-1}\ .
\]
Then $c$ is well-defined for $w\in D$, its image $c(w)$ belongs to $T_\Omega$ and
$c$ yields a biholomorphic diffeomorphism from $D$ into $T_\Omega$. Now for $u\in U$, define
\begin{equation}\label{Cayley}
\widetilde u = c\circ u \circ c^{-1}\ .
\end{equation}
Then $u\in \widetilde U$ and the map $u\longmapsto \widetilde u$ is a isomorphism of 
$U$ to $\widetilde U$. Notice that
\begin{equation}\label{Jacu}
D\widetilde u (ie) = u\ .
\end{equation}
This is obtained by the chain rule applied to \eqref{Cayley} and the fact that $Dc(0) = 2i\id_{\mathbb J}$.
\section{Vector-valued holomorphic representations and Laplace transform}

Let $(\sigma, V_\sigma)$ be a finite dimensional (holomorphic) irreducible representation of $\mathbb L$ (equivalently of $U$ or $L$) and choose an inner product $(\,.\,,\,.\,)_{V_\sigma}$ on $V_\sigma$ which is invariant under the action of $U$. Let $\mathcal O_\sigma$ be the Montel space of $V_\sigma$-valued holomorphic functions on $T_\Omega$. The following formula defines a representation of $G$ on $\mathcal O_\sigma$
\begin{equation}\label{defpi}
\pi_{\sigma}(g)F(z) = \sigma\big(J(g^{-1},z)\big)^{-1} F\big(g^{-1}(z)\big)\ .
\end{equation}
For $F,G\in \mathcal O_\sigma$ let
\[(F,G)_\sigma = \int_{T_\Omega} \left(\sigma(P(y)^{-1} F(z), G(z)\right)_{V_\sigma} \,d_*z
\]
where $d_*z =( \det y)^{-\frac{2n}{r}} \,dx\,dy$ is the $G$-invariant measure on $T_\Omega$, and let 
\[\mathcal H_\sigma = \{ F\in \mathcal O_\sigma, (F,F)_\sigma  <+\infty\}\ .
\]

When this space is not reduced to $\{0\}$, it is a Hilbert space, stable by the action of $G$ and $\pi_\sigma$ yields a unitary representation of $G$. Then the evaluation map at any $z\in T_\Omega$
\[\mathcal H_\sigma \ni F\longmapsto F(z) \in V_\sigma
\]
is easily shown to be continuous. Define for $z,w\in T_\Omega$ the \emph{reproducing kernel} of $\mathcal H_\sigma$ by
\[\mathcal Q_\sigma(z,w) = E_zE_w^*\quad \ .
\]

\begin{proposition} Assume that $\mathcal H_\sigma \neq \{0\}$. Then 
\[\mathcal Q_\sigma(z,w) = c\,\sigma\left(P\left(\frac{z-\overline w}{2i}\right)\right)
\]for some constant $c>0$.
\end{proposition}
See \cite{c95} for a proof. Denote by $\Herm^+(V_\sigma)$ the space of  positive semi-definite operators on $V_\sigma$. 

\begin{proposition}\label{dRsigma} Suppose that $\mathcal H_\sigma\neq \{0\}$. There exists a unique $\Herm^+(V_\sigma)$-valued measure $dR_\sigma$ on $\overline \Omega$ such that
\begin{equation}\label{defQsigma}
\mathcal Q_\sigma(z,w) = \int_{\overline{\Omega} }e^{-(\frac{z-\overline w}{2i}\,\vert\,v)} dR_\sigma(v)\ .
\end{equation}
Moreover, the measure $dR_\sigma$ satisfies
\begin{equation}\label{invR}
\forall \ell \in L,\qquad dR_\sigma(\ell\, .\,) ={ \sigma(\ell)^*}^{-1}\, dR_\sigma(\, .\,)\, \sigma(\ell)^{-1}\ ,
\end{equation}
\begin{equation}\label{normR}
\int_{\overline \Omega} e^{-\tr v} dR_\sigma(v) = \Id_{V_\sigma}\ .
\end{equation}
\end{proposition}
For a proof see \cite{c95}.

Let $\mathcal L_\sigma$ be the space of mesurable functions $f:\overline\Omega 
\longrightarrow V_\sigma$ which satisfy
\begin{equation}\label{defLsigma}
 \int_{\overline \Omega} \big(dR_\sigma(2v) f(v),f(v)\big)_{V_\sigma} <+\infty\ .
\end{equation}
After identifying two functions which 	are equal $dR_\sigma$-a.e., $\mathcal L_\sigma$ becomes a Hilbert space for the inner product
\begin{equation}
(f,g)_{\mathcal L_\sigma} = \int_{\overline \Omega} \big(dR_\sigma(2v)f(v),g(v)\big)_{V_\sigma}\ .
\end{equation}
Define the (modified) \emph{Laplace transform} $\mathcal F_\sigma$ by
\begin{equation}\label{defLaplace}
\mathcal F_\sigma f(z) = \int_{\overline \Omega} e^{i\,(z\,\vert\, v)} dR_\sigma(2v) f(v)\ .
\end{equation}
\begin{proposition}
Assume that $\mathcal H_\sigma\neq \{0\}$. Then the Laplace transform $\mathcal F_\sigma$ yields an isometry from $\mathcal L_\sigma$ onto $\mathcal H_\sigma$.
\end{proposition} 

For the proof, see  \cite{c95}, Theorem 3.5.

The results presented so far can be summarized as follows : to each representation $\sigma$ such that $\mathcal H_\sigma$ is not reduced to $\{0\}$, there corresponds an $\Herm^+(V_\sigma)$-valued measure  $dR_\sigma $ on $\overline \Omega$ which satisfies \eqref{invR} and \eqref{normR}. In some sense, there is a converse construction, starting form the measure $dR_\sigma$ and constructing the Hilbert space $\mathcal H_\sigma$.  

\begin{proposition} Let $\sigma$ be a holomorphic irreducible representation of $\mathbb L$ and assume that there exists  an $\Herm^+(V_\sigma)$-valued measure  $dR_\sigma $ on $\overline \Omega$ which satisfies \eqref{invR} and \eqref{normR}. Define the space $\mathcal L_\sigma$ by the condition \eqref{defLsigma} and let $\widetilde {\mathcal H}_\sigma$ be the image of $\mathcal L_\sigma$ by the Laplace transform \eqref{defLaplace}. Then $\widetilde {\mathcal H}_\sigma$, equipped with the inner product given by
\[(\mathcal F_\sigma f, \mathcal F_\sigma g)  = (f,g)_{\mathcal L_\sigma}
\]
 is a Hilbert space of holomorphic functions in $T_\Omega$, which admits the reproducing kernel given by
\[\mathcal Q_\sigma(z,w) = \int_{\overline \Omega} e^{-(\frac{z-\overline w}{2i}\,\vert \, v)} dR_\sigma(v)\ .
\]
The space $\widetilde {\mathcal H}_\sigma$ is stable under the action of $G$ given by \eqref{defpi} and $(\pi_\sigma,\widetilde{\mathcal H}_\sigma)$ is an irreducible unitary representation of $G$.
\end{proposition}

For the proof see again \cite{c95}, Theorem 3.5.

\noindent
{\bf Remark.} For $z, w\in T_\Omega$,
\[\Re\left(\frac{z-\overline w}{2i}\right) = \frac{\Im(z)+\Im(w)}{2}\in \Omega,
\]
and on the diagonal
\[\Re\left(\frac{z-\overline z}{2i}\right) = \Im(z)\in \Omega\ .
\]
Hence $\det\left(\frac{z-\overline w}{2i}\right)\neq 0$ on $T_\Omega\times T_\Omega$,
and there exists a unique determination of $\displaystyle \log \det\left(\frac{z-\overline w}{2i}\right)$ on $T_\Omega\times T_\Omega$ which coincides on the diagonal $\{ w=z\}$ with
$\ln \left(\det (\Im z)\right)$. With this choice, define for $\mu\in \mathbb C$, 
\[\det\left(\frac{z-\overline w}{2i}\right)^\mu = e^{\mu \log\left(\det\left(\frac{z-\overline w}{2i}\right)\right)}\ ,
\]
and let
\[\sigma_\mu(\ell) = \chi(\ell)^{-\frac{\mu}{2} }\sigma(\ell)\ .
\]
This makes sense for $\ell\in L$ and defines a representation of $L$, which can also be considered as a representation of the universal covering of $\mathbb L$.
It may be used to define a holomorphic representation $\pi_{\sigma_\mu}$ of the universal covering of $G$ still using formula \eqref{defpi}. Thanks to \eqref{chiP}, the corresponding reproducing kernel is given by
\[\mathcal Q_{\sigma, \mu}(z, w)= \det\left(\frac{z-\overline w}{2i} \right)^{-\mu}\mathcal Q_\sigma(z,w)\ .
\]
This opens the possibility of studying the existence of  $dR_\sigma$ by using techniques of \emph{analytic continuation}.

\section{Polynomial representations of $L$ and the associated holomorphic discrete series}
Let $\mathcal P$ be the space of holomorphic polynomials on $\mathbb J$. The group $\mathbb L$ acts naturally on $\mathcal P$ by
\[\pi(\ell)p\,(z) = p(\ell^{-1} z)\ ,
\]
for $p\in \mathcal P$ and $\ell\in \mathbb L$. It may also be regarded as a representation of $U$ or of $L$, and clearly the decomposition of $\mathcal P$ into invariant minimal subspaces is the same for the three different points of view.

Fix a Jordan frame $(c_1,c_2,\dots, c_r)$ of $J$. For $1\leq k\leq r$ let
\[e_k= c_1+c_2+\dots=c_k ,\qquad
J_k = J(e_k,1)\ .
\]
Then $J_k$ is simple Jordan algebra with neutral element $e_k$. Notice that 
for any $k, 0\leq k\leq r$, \[\Omega_k = L\, e_k\ .
\]

Let ${det}_k$ be its determinant and let $p_k$ be the orthogonal projector of $J$ unto $J_k$. Form
\[\Delta_k(x) ={ \det}_k\,(p_k x)
\]
the \emph{$k$-th principal minor}.

A multiindex $\boldsymbol m = (m_1,m_2,\dots, m_r)$ where  $m_j\in \mathbb Z, 1\leq j\leq r$ is said to be \emph{positive} if $m_1\geq m_2\geq \dots \geq m_r\geq 0$. The set of positive multiindices is denoted by $\boldsymbol {\mathcal M}_+$. For $\boldsymbol m\in \boldsymbol {\mathcal M}_+$ let
\[\Delta_{\boldsymbol m}(x) = \Delta_1(x)^{m_1-m_2} \Delta_2(x)^{m_2-m_3} \dots \Delta_r(x)^{m_r}\ .
\]
and let $\mathcal P_{\boldsymbol m}$ be the subspace of $\mathcal P$ generated by  $\{\pi(\ell)\Delta_{\boldsymbol m}, \  \ell \in L\}$. 

\begin{proposition}\label{decompP}
 The subspaces $\mathcal P_{\boldsymbol m}$ are mutually inequivalent irreducible subspaces under the action of $L$ and $\mathcal P$ is the direct orthogonal sum
\[\mathcal P= \bigoplus_{\boldsymbol m\in \boldsymbol {\mathcal M}_+} \mathcal P_{\boldsymbol m}\ .
\]
\end{proposition}
Denote by $ \pi_{\boldsymbol m}$ the restriction of $\pi$ to the subspace $\mathcal  P_{\boldsymbol m}$ will be denoted by $\pi_{\boldsymbol m}$.

Let $\boldsymbol m =(m_1,m_2,\dots, m_r)\in \boldsymbol{\mathcal M_+}$ and let $\mathcal O_{\boldsymbol m} = \mathcal O(T_\Omega, \mathcal P_{\boldsymbol m})\simeq \mathcal O(T_\Omega)\otimes \mathcal P_{\boldsymbol m}$ be the space of $\mathcal P_{\boldsymbol m}$-valued holomorphic functions on $T_\Omega$. Then $G$ acts on $\mathcal O_{\boldsymbol m}$ by 
\[\pi_{\boldsymbol m}(g) F(z) =\sigma_{\boldsymbol m} (J(g^{-1},z)\big)^{-1}
F\big(g^{-1}(z)\big)\ .
\]
Form the corresponding Hilbert space $\mathcal H_{\sigma_{\boldsymbol m}}$ and the corresponding representation $\pi_{\sigma_{\boldsymbol m}}$ of $G$ on $\mathcal H_{\sigma_{\boldsymbol m}}$. Let further
\[\mathcal Q_{\boldsymbol m} (z,w) = \sigma_{\boldsymbol m} \left(P\left(\frac{z-\overline w}{2i}\right)\right).
\]
When $\mathcal H_{\boldsymbol m}\neq \{ 0\}$, this is (up to a positive scalar) the reproducing kernel of $\mathcal H_{\boldsymbol m}$, and there exists a $\Herm^+(\mathcal P_{\boldsymbol m})$-valued measure $dR_{\boldsymbol m}$, supported in $\overline\Omega$, such that \eqref{defQsigma} is satisfied. The question we address is to calculate $dR_{\boldsymbol m}$, using  the two properties \eqref{invR} and \eqref{normR} which allow in principle to calculate a formal solution. It remains to test the positivity of the  solution to conclude.

In this case, the remark can be exploited as follows.
Let $\mu\in \mathbb N$, and consider the map $I_\mu : \mathcal P\rightarrow \mathcal P$ given by
\[(I_\mu\,p\,)(z) = (\det z)^\mu  p(z)
\]
and define the representation $\sigma_{\boldsymbol m, \mu}$ of $\mathbb L$ on $\mathcal P_{\boldsymbol m}$ given by
\[\sigma_{\boldsymbol m, \mu}(\ell)\, p = \chi(\ell)^{-\mu}\  (p\circ \ell^{-1})\ .
\]
\begin{proposition} The operator $I_\mu$ maps $\mathcal P_{\boldsymbol m}$ onto $\mathcal P_{{\boldsymbol m}+\mu}$ and intertwines $\sigma_{\boldsymbol m, \mu}$ and $\sigma_{{\boldsymbol m}+\mu}$.
\end{proposition}

\begin{proof}
Now first
\[(I_\mu\,\Delta_{\boldsymbol m}) (z) = \Delta_1(x)^{m_1-m_2}\dots \Delta_r^{m_r+\mu}(z)
\]
\[= \Delta_{m_1+\mu,m_2+\mu,\dots, m_{r-1}+\mu, m_r+\mu}(z)= \Delta_{\boldsymbol m+\mu}(z) \in \mathcal P_{\boldsymbol m+\mu}
\]
Next for $p\in \mathcal P_{\boldsymbol m}$
\[I_\mu(\sigma_{\boldsymbol m,\mu}(\ell)p) (z) = \chi(\ell)^{-\mu}p(\ell^{-1} z) \Delta_r^\mu(z) = p(\ell^{-1}z)\Delta_r^\mu(\ell^{-1} z)= (I_\mu p)(\ell^{-1} z)\ .
\]
Combining both formulas,
\[I_\mu(\sigma_{{\boldsymbol m}, \mu}(\ell) \Delta_{\boldsymbol m}) = \Delta_{{\boldsymbol m}+\mu}\circ \ell^{-1} = \sigma_{{\boldsymbol m} +\mu(\ell)} I_\mu \Delta_{\boldsymbol m}\ .
\]
\end{proof}
As observed previously, this can be extended for $\mu$ a complex number, by considering the universal covering of $\mathbb L$.

The strategy is to consider a multiindex $\boldsymbol m$ with last index $m_r=0$ and study the family of kernels $\mathcal Q_{\sigma_{\boldsymbol m+\mu}},\mu\in \mathbb C$ by analytic continuation. 

The scalar case (i.e. $\boldsymbol m=(m,m,\dots,m)$) is treated in \cite{fk} ch. XIII and  the study leads to the so-called \emph{Wallach set} (see \cite{w} and \cite{rv} for original proofs).

For $x\in \Omega$, $\Delta_j(x)>0$. For $\boldsymbol s = (s_1,s_2,\dots,s_r)\in \mathbb C^r$, write
\[\Delta_{\boldsymbol s} (x) = \Delta_1(x)^{s_1-s_2}\Delta_2(x)^{s_2-s_3}\dots \Delta_r(x)^{s_r}\ .
\]
The \emph{Gamma function} of the cone $\Omega$ is defined by
\[\Gamma_\Omega(\boldsymbol s)= \int_\Omega e^{-\tr x} \Delta_{\boldsymbol s}(x) d^*x
\]
The integral is absolutely convergent for $\Re s_j > (j-1)\frac{d}{2}, j=1,2,\dots,r$ and can be meromorphically continued to $\mathbb C^r$. Moreover,
\begin{equation}\Gamma_\Omega(\boldsymbol s)= (2\pi)^{\frac{n-r}{2}}\prod_{j=1}^r \Gamma\left(s_j-(j-1)\frac{d}{2}\right)\ , 
\end{equation}
see \cite{fk} Ch. VII.
\section{The case of $\mathcal P_{1,0,\dots,0}$}

In this section, we achieve the computation of $dR_\mu$ for  the case corresponding to the family  of indices $\boldsymbol m = (1,0,\dots,0) + \mu, \mu\in \mathbb C$. \footnote{The case where $J$ is of rank $1$, i.e. $J=\mathbb R$, is excluded as there are only scalar scalar valued cases.}

The space $\mathcal P_{1,0,\dots,0}$ is the dual space $\mathbb J'$ of complex linear forms on $\mathbb J$. Using the duality (over $\mathbb C$) given by
\[(x,y) \longmapsto \tr xy
\]
it is convenient to identify $\mathbb J'$ with $\mathbb J$. The corresponding representation $\sigma$ of $\mathbb L$ is then given by
\[\sigma(\ell) v ={ \ell^t}^{-1}v\ .
\]
The corresponding $U$-invariant Hilbertian inner product is given by
\[(x\vert y) = \tr x\overline y\ .
\]
The following (folklore) lemma will be useful later.
\begin{lemma}\label{invquadJ}
Let $J$ be a simple Euclidean Jordan algebra of rank $r\geq 2$. Let $Q$ be a Hermitian operator on $\mathbb J$ which is invariant under $K$. Then there exists $\alpha,\beta\in \mathbb R$ such that for any $v, w\in \mathbb J$
\[Qv = \alpha v + \beta (v\vert e)\, e .
\]
\end{lemma}
\begin{proof} Let $q(v,w) = (Qv,w)$ be the associated Hermitian form on $\mathbb J$.
Let $(c_1,c_2,\dots, c_r)$ be a Jordan frame of $J$. Let $\mathbb A = \bigoplus_{j=1}^r \mathbb C c_j$ and let $q_{\mathbb A}$ be the restriction of $q$ to $\mathbb A$. Given a permutation $\sigma$ of $\{1,2,\dots,r\}$ there exists an element $k_\sigma$ of $K$ such that $k_\sigma c_j = c_{\sigma(j)}$. Hence $q_\mathbb A$ is invariant under $\mathfrak S_r$. The space $\mathbb A\simeq \mathbb C^r$ decomposes under the action of $\mathfrak S_r$ as
\[\mathbb A = \mathbb C e\oplus \mathbb A_0= \mathbb C e \bigoplus \left\{a=\sum_{j=1}^r a_jc_j,\  \sum_{j=1}^r a_j= 0 \right\}\ ,
\]
and $\mathbb A_0$ is irreducible under the action of $\mathfrak S_r$. Hence the space of $\mathfrak S_r$-invariant Hermitian forms on $\mathbb A$ is of dimension 2. The two  Hermitian forms $\sum_{j=1}^r a_j\overline b_j$ and $(\sum_{j=1}^r a_j)(\sum_{j=1}^r \overline b_j)$ on $\mathbb C^r$ are $\mathfrak S_r$-invariant and they are linearly independent. Hence there exists two real numbers $\alpha,\beta$ such that 
\[q_\mathbb A (\sum_{j=1}^r a_j c_j, \sum_{j=1}^r b_j c_j) = \alpha\left(\sum_{j=1}^r a_j\overline b_j\right) + \beta \left(\sum_{j=1}^r a_j\right)\left(\sum_{j=1}^r \overline b_j\right)\ .
\]
As a consequence, the two Hermitian forms $q$ and $\alpha \tr(x\overline x)+ \beta \tr(x) \tr(\overline x)$ on $\mathbb J$ are invariant by $K$ and coincide on $\mathbb A$. By the spectral theorem for $J$, they coincide on $J$  and hence also on $\mathbb J$. The conclusion follows
\end{proof} 

Recall that  $\sigma_\mu$ is the representation of $L$ given by
\[\sigma_\mu(\ell) = \chi (l)^{-\frac{\mu}{2}} \sigma(\ell)\ .
\]
The measure $dR_\mu$ on $\overline \Omega$ should satisfy
\begin{equation}\label{invR'}
\text{for all } \ell \in L,\qquad dR_\mu(\ell .) = \chi( \ell)^{\mu} \ell\, dR_\mu(.)\, \ell^t
\end{equation}
and
\begin{equation}\label{normR'}
\int_{\overline\Omega} e^{-\tr v} dR_\mu(v) = \Id_\mathbb J\ .
\end{equation}
\begin{theorem}\label{rmuOmega}
 Let $\Re \mu>(r-1)\frac{d}{2}$.
 Let $dR_\mu$ be the $Herm(\mathbb J)$-valued measure on $\Omega$ defined by
\begin{equation}\label{dRmudef}
 dR_\mu(y) =  \frac{1}{\mu(\mu+1)(\mu-\frac{d}{2})\Gamma_\Omega(\mu)}\left(\mu P(y) -\frac{d}{2} p_y\right)(\det y)^\mu d^*y
\end{equation}
where for $y\in J$, $p_y$ is the operator defined by $p_y\,v\, = (v\vert y) y$.

Then
\[\int_\Omega e^{-\tr y} \,dR_\mu(y) =  \Id_\mathbb J\ .
\]
\end{theorem}

Before giving a proof of the theorem, we present a heuristic approach which led us to the formula. Assume that the measure $dR_\mu$ is absolutely continuous with respect to the $L$-invariant measure $(\det y)^{-\frac{n}{r}} dy$ on $\Omega$, and let
\[dR_\mu(y) = r_\mu(y)(\det y)^{-\frac{n}{r}} dy
\]
be its expression, where $r_\mu$ is now a $\Herm(\mathbb J)^+$-valued function on $\Omega$. The function $r_\mu$ has to satisfy
\begin{equation}\label{covr}
\forall \ell\in L, \forall x\in \Omega,\qquad r_\mu(\ell x) =\chi(l)^\mu\, l\,r_\mu(x)\,l^t
\end{equation}
and
\begin{equation}\label{normr}
\int_\Omega e^{-\tr x} r_\mu (x) \, dx^*= \id_\mathbb J\ .
\end{equation}

\begin{lemma} Let $r_\mu$ a $\Herm(\mathbb J)$-valued function on $\Omega$ which satisfies \eqref{covr}. Then there exist two real numbers  $\alpha(\mu), \beta(\mu)$ such that
\begin{equation}\label{alphabeta}
r_\mu(y) = (\det y)^\mu \big( \alpha(\mu) P(y) +\beta(\mu)\, p_y\big) \ .
\end{equation}

\end{lemma}
\begin{proof}
Letting $\ell \in K$ and $v=e$, \eqref{covr} implies that for any $k\in K$
\[k\,r_\mu(e)= r_\mu(e)\, k\ .
\]
Now use Lemma \ref{invquadJ} to conclude that there exists two real numbers $\alpha(\mu)$ and $\beta(\mu)$ such that
\[r_\mu(e)v = \alpha(\mu) v + \beta(\mu) \tr(v) e\ .
\]
Further, let $y\in \Omega$ and use again \eqref{covr} with $x=e$ and $\ell = P(y^{1/2})$ to get
\[
r_\mu(y) = (\det y)^\mu \big( \alpha(\mu) P(y) +\beta(\mu)\, p_y\big) \ .
\]

\end{proof}
To prove Theorem \ref{rmuOmega}, it remains to compute $\alpha(\mu)$ and $\beta(\mu)$, so that \eqref{normr} is satisfied. Consider the four following functions on $J$
\[\Tr P(x),\quad  \Tr p_x,\quad (P(x)e\vert e),\quad   (p_x(e)\vert e)\ .
\]
They are homogeneous polynomials of degree $2$ and they are invariant by $K$.
The space $\mathcal P_2$ of homogeneous polynomials of degree $2$ decomposes under the action of $L$ as
\[\mathcal P_2 = \mathcal P_{2,0,\dots,0} \oplus \mathcal P_{1,1,0,\dots,0}\ ,\]
as a consequence of Proposition \ref{decompP}.
Now each space $\mathcal P_{\boldsymbol m}$ contains a unique (up to a scalar) $K$-invariant vector (see \cite{fk} Proposition XI.3.1), so that the four polynomials are linear combination of the two $K$-invariant polynomials in $\mathcal P_{2,0}$ and $\mathcal P_{1,1}.$
\begin{lemma}
 Consider the following polynomials
\[p_{2,0}(x) = \frac{d}{2}\, (\tr x)^2 +\tr(x^2)\ ,\quad p_{1,1}(x) = (\tr x)^2 -\tr(x^2)\ 
\]
Then $p_{2,0}$ (resp. $p_{1,1}$) is the unique (up to a scalar) $K$-invariant element in  $\mathcal P_{2,0,\dots,0}$ (resp.  $\mathcal P_{1,1,0,\dots,0}$).
\end{lemma}

\noindent
For a proof, see \cite{fk} ch XI. Exercice 2.

\begin{lemma} \label{Trtr} Let $J$ be a simple Euclidean Jordan algebra of rank $r\geq 2$. Then the following identities hold 
\begin{equation} \label{TrP}
\Tr P(x) =  \frac{2}{d+2}\ \Big(p_{2,0}(x)\,+\,\frac{d^2}{4}\,p_{1,1}(x) \Big)
\end{equation}
\begin{equation}\label{trxsq}
(P(x)e\vert e) = \Tr p_x =  \frac{2}{d+2} \Big(p_{2,0}(x))-\frac{d}{2}\big(p_{1,1}(x)\Big)
\end{equation}
\begin{equation}\label{trx)sq}
(p_xe\vert e)=  \frac{2}{d+2} \big(p_{2,0}(x)+p_{1,1}(x)\big)\ .
\end{equation}
\end{lemma}
\begin{proof} First observe that 
\[(P(x)e\vert e) = ( x^2\vert e) = \tr(x^2),\quad  \Tr p_x= (x\vert x)= \tr(x^2),\]\[  ( p_xe\vert e)=(x\vert e) (x\vert e)= (\tr x)^2\ .
\]
Hence \eqref{trxsq} and \eqref{trx)sq} express the change of basis from $\{\tr (x^2), (\tr x)^2\}$ to $\{p_{2,0}, p_{1,1}\}$. For \eqref{TrP}, it suffices to evaluate both sides  on  $x=e$ and on $x=c$ a primitive idempotent ($e$ and $c$ are not proportional, as $\rank J\geq 2$). Now
 \[\tr(e) = \tr(e^2) = r,\quad \Tr P(e) = n,\quad \tr(c) = \tr(c^2)=1,\quad \Tr P(c)=1
 \]
 and the verification follows  by elementary computation. 
 \end{proof}
 \begin{lemma}\label{I20I11}
 For $\Re \mu >(r-1)\frac{d}{2}$,
 \begin{equation}
 I_{2,0}= \int_\Omega e^{-\tr x} p_{2,0}(x)(\det x)^\mu d^*x = r\Big(1+\frac{rd}{2}\Big)\mu(\mu+1)\Gamma_\Omega(\mu)
 \end{equation}
 \begin{equation}
 I_{1,1}= \int_\Omega e^{-\tr x} p_{1,1}(x)(\det x)^\mu d^*x =r(r-1)\,\mu\,(\mu-\frac{d}{2})\Gamma_\Omega(\mu)
 \end{equation}
 \end{lemma}

\begin{proof} Let $p\in \mathcal P_{\boldsymbol m}$. For $\Re \mu> (r-1)\frac{d}{2}$,
\begin{equation}
\int _\Omega e^{-\tr y} p(y) (\det y)^\mu \,d^*x = \Gamma_\Omega(\boldsymbol m+\mu)\, p(e)\ .
\end{equation}
See \cite{fk} Lemma XI.2.3.
Apply now this formula and for $\boldsymbol m =(2,0,\dots,0)$
\[\int_\Omega e^{-\tr x} p_{2,0}(x)\,(\det x)^\mu\, d^*x=
\]
\[= (2\pi)^{\frac{n-r}{2}}r\left(1+\frac{rd}{2}\right) \Gamma(2+\mu) \prod_{j=2}^r \Gamma\left(\mu-(j-1)\frac{d}{2}\right)\ 
\]
\[=r\left(1+\frac{rd}{2}\right)(\mu+1)\,\mu \,\Gamma_\Omega(\mu)
\]

and for $\boldsymbol m=(1,1,0,\dots,0)$  to obtain
\[\int_\Omega e^{-\tr x}p_{1,1}(\det x)^\mu\, d^*x
\]
\[=(2\pi)^{\frac{n-r}{2}}r(r-1) \Gamma(\mu+1) \Gamma\left(\mu+1-\frac{d}{2} \right)\prod_{j=3}^r \Gamma\left(\mu-(j-1)\frac{d}{2}\right)\]
\[=r\,(r-1)\,\mu\,\left(\mu-\frac{d}{2}\right) \,\Gamma_\Omega(\mu)\ .
\]
\end{proof}

\begin{proposition}
 For $\Re \mu >(r-1)\frac{d}{2}$
\begin{equation} 
\begin{split}&\hskip 2.5cm\int_\Omega e^{-\tr x}\Tr(P(x)) (\det x) ^\mu d^*x \\ &=\frac{2}{d+2}\Gamma_\Omega(\mu)r\mu\Big((1+\frac{rd}{2})(\mu+1)+\frac{d^2}{4}(r-1)(\mu-\frac{d}{2})\Big)
\end{split}
\end{equation}
\begin{equation}
\begin{split}&\hskip 2.5cm\int_\Omega e^{-\tr x}\Tr(p_x) (\det x) ^\mu d^*x\\=&\frac{2}{d+2}\Gamma_\Omega(\mu)r\mu\Big((1+\frac{rd}{2})(\mu+1)-\frac{d}{2}(r-1)(\mu-\frac{d}{2})\Big)
\end{split}
\end{equation}
\begin{equation}
\begin{split}&\hskip 2.5cm\int_\Omega e^{-\tr x} ( P(x)e\vert e) (\det x)^\mu d^*x 
\\ =&\frac{2}{d+2}\Gamma_\Omega(\mu)r\mu\Big( (1+\frac{rd}{2})(\mu+1)-\frac{d}{2} (r-1) (\mu-\frac{d}{2})\Big)
\end{split}
\end{equation}
\begin{equation}
\begin{split}
&\hskip 2.5cm\int_\Omega e^{-\tr x} (p_x e\vert e) (\det x)^\mu d^*x 
\\
&=\frac{2}{d+2}\Gamma_\Omega(\mu)r\mu\Big( (1+\frac{rd}{2})(\mu+1)+(r-1) (\mu-\frac{d}{2})\Big)\ .
\end{split}
\end{equation}
\end{proposition}
\begin{proof} Combine Lemma \ref{Trtr} and Lemma \ref{I20I11} to get the results.
\end{proof}
As $\Tr(\Id_{\mathbb J}) = n$ and $(\Id_{\mathbb J} e,e) = \tr e = r $, $\alpha(\mu), \beta(\mu)$ have to satisfy
\[\begin{pmatrix}
(2+rd)(\mu+1)+(r-1)\frac{d^2}{2}(\mu-\frac{d}{2})&(2+rd)(\mu+1)-(r-1)d(\mu-\frac{d}{2})\\ \\
(2+rd)(\mu+1)-d(r-1)(\mu-\frac{d}{2})&(2+rd)(\mu+1)+2(r-1)(\mu-\frac{d}{2})
\end{pmatrix}
\begin{pmatrix}
\alpha(\mu)\\ \\ \beta(\mu)  
\end{pmatrix}
\]
\[
= \frac{d+2}{r\mu\Gamma_\Omega(\mu)}\begin{pmatrix} n\\ \\ r\end{pmatrix}\ .
\]
Now observe that the four coefficients of the matrix vanish identically for $d=-2$. Hence the system can be rewritten as
\[\begin{pmatrix}\left(\frac{(r-1)d}{2}+1\right)\mu-(r-1)\frac{d^2}{4}+\frac{(r-1)d}{2}+1&&\mu+\frac{(r-1)d}{2}+1\\ \\ \mu+\frac{(r-1)d}{2}+1& &r\mu +1
\end{pmatrix}
\begin{pmatrix}\alpha(\mu)\\ \\ \beta(\mu)\end{pmatrix}\]\[= \frac{1}{r\mu\Gamma_\Omega(\mu)}\begin{pmatrix}n\\ \\ r\end{pmatrix}
\]
The determinant of the matrix is a polynomial of degree 2 in $\mu$, and it is easily seen that it vanishes for $\mu=-1$ and $\mu= \frac{d}{2}$. Next, the coefficient of degree $2$ is equal to
\[\left\vert\begin{matrix}\left(\frac{(r-1)d}{2}+1\right)&& 1\\ \\ 1&r
 \end{matrix} \right\vert= n-1\ .
\]
Hence the determinant of the matrix of the system  is equal to
\[(n-1) (\mu+1)\left(\mu-\frac{d}{2}\right)\ .
\]
The solutions of the system are given by
\[\begin{pmatrix} \alpha(\mu)\\ \\ \beta(\mu) \end{pmatrix}  = 
 \frac{1}{r(n-1)\mu(\mu+1)(\mu-\frac{d}{2})\Gamma_\Omega(\mu)}\]\[
 \begin{pmatrix}r\mu+1&&-\mu-\frac{(r-1)d}{2}-1 \\ \\ 
 -\mu-\frac{(r-1)d}{2}-1&& \left(\frac{(r-1)d}{2}+1\right)\mu-(r-1)\frac{d^2}{4}+\frac{(r-1)d}{2}+1
 \end{pmatrix}
 \begin{pmatrix}n\\ \\ r
 \end{pmatrix}
 \]
 Notice that
 \[n(r\mu+1)-r\mu -(r(r-1)\frac{d}{2}+r)= (n-1)r\mu
 \]
 
 whereas
 \[-n\mu-\frac{(r-1)d}{2}\,n-n+n\mu-r(r-1)\frac{d^2}{4}+n=-(r-1)\frac{d}{2}\left(n+r\frac{d}{2}\right)
 \]
 \[ =-\frac{d}{2} \left(n(r-1)+r(r-1)\frac{d}{2} \right) = -\frac{d}{2}\left(n(r-1)+(n-r) \right)
 \]
 \[= -\frac{d}{2}(n-1)r\ .
 \]
 Hence
 \[\begin{pmatrix} \alpha(\mu)\\ \\ \beta(\mu) \end{pmatrix}  = 
 \frac{1}{\mu(\mu+1)(\mu-\frac{d}{2})\Gamma_\Omega(\mu)}\begin{pmatrix} 
 \mu\\ \\ -\frac{d}{2}
 \end{pmatrix} \ .
 \]
 Now, following  \eqref{alphabeta} set
 \[r_\mu(x) = \frac{1}{\mu(\mu+1)(\mu)-\frac{d}{2})\Gamma_\Omega(\mu)}\left(\mu P(x)-\frac{d}{2} p_x\right)(\det x)^\mu
 \]
 and  this achieves the proof of Theorem \ref{rmuOmega}.
\smallskip

We now discuss for which values of $\mu$ is $ r_\mu$  $\Herm^+(\mathbb J)$-valued. 

\begin{lemma}\label{lemmapyPy}
 Let $y\in \overline \Omega$. Then for any $v\in \mathbb J$
\begin{equation}\label{pyPy}
\big(p_y \,v\vert v\big)\leq r \big(P(y)v\vert v\big)\ .
\end{equation}
Moreover, if $y\in \Omega$ and $v\neq 0$, then equality in \eqref{pyPy} holds if and only if $y=t e, t\in \mathbb R^+$ and $v= \lambda e, \lambda\in \mathbb C^*$.
\end{lemma}

\begin{proof}
We may assume that $y\neq 0$. The space $\mathbb J$ splits as $\mathbb J = \mathbb C y \,\oplus (y)^\perp$. Both subspaces are invariant under $p_y$. Moreover $P(y)$ preserves $\mathbb C y$ and as it is a selfadjoint operator, it also preserves $(y)^\perp$. Hence to prove the inequality \eqref{pyPy}, it is enough to prove the inequality separately on both subspaces. The inequality is trivial on $(y)^\perp$, so that it is enough to prove it on $\mathbb C y$. Now
\[\big(p_y\,y\vert y\big) = (y\vert y)^2,\qquad \big(P(y)\,y\vert y\big) =(y^3\vert y) =(y^2\vert y^2)\ .
\]
By the Cauchy-Schwarz inequality,
\begin{equation}\label{Cauchyeq}
(y\vert y) = (y^2\vert e) \leq  \big(y^2\vert y^2)^{\frac{1}{2}}(e\vert e)^{\frac{1}{2}},
\end{equation}
 so that 
 \[(y\vert y)^ 2\leq r (y^2\vert y^2)
 \]
 and hence \eqref{pyPy} is also satisfied on $\mathbb C y$. 

Assume now that $y\in \Omega$. Then as $P(y)$ is positive-definite, the inequality \eqref{pyPy} is strict on $(y)^\perp$. So equality can occur only if  $v=\mu y$ for some $\mu\in \mathbb C, \mu\neq 0$. Conversely, equality in \eqref{pyPy} implies
\[( y\vert y)^2 = r( y^2\vert y^2) ,
\]
hence corresponds to the case of equality in the Cauchy-Schwarz inequality \eqref{Cauchyeq} and implies $y=te$ for some  $t>0$. This completes the proof.
\end{proof}
The next proposition is just an immediate consequence of the previous lemma.
\begin{corollary} Let $\mu\in \mathbb R$. For $y\in \overline \Omega$, the operator $r_\mu(y)$ is positive semi-definite if and only if $\mu \geq\frac{rd}{2}$.
\end{corollary}
\section {The singular cases}

To prove Theorem \ref{Wallachset}, it remains (cf \cite{c95}) to look to the cases $\mu=k\frac{d}{2}, 0\leq k\leq r-1$. If $k=0$, i.e. for the orbit $\mathcal O^{(0)} = \{ 0\}$, there is clearly no solution to the condition \eqref{invR'}. The next section will treat the case $k=1$. We examine in this section the case where $ 2\leq k\leq r-1$, assuming that $r\geq 3$. The next lemma is a generalization of Lemma \ref{lemmapyPy}.
\begin{lemma}\label{lemmapyPy2}
 Let $1\leq k\leq r$ and let $y\in \Omega^{(k)}$. Then for all $v\in \mathbb J$
\begin{equation}\label{pyPy2}
\big(p_y\,v\,\vert\,v\big) \leq k\big(P(y)v\,\vert v\big)\ .
\end{equation}
\end{lemma}
\begin{proof}
Let $c$ be an idempotent of rank $k$. Decompose $\mathbb V$ as $\mathbb C c\oplus (c)^\perp$. Both subspaces are invariant by $p_c$ and by $P(c)$. Hence it suffices to verify \eqref{pyPy2} separately on each subspace. The inequality being trivial on $(c)^\perp$, it suffices to verify it on $\mathbb C c$. So let $v=\lambda c$, with $\lambda\in \mathbb C$. Then
\[\big(p_c(v)\vert v\big) = \vert\lambda\vert^2(c\vert c),\qquad \big(P(c)v\vert v\big) = \vert \lambda\vert^2\ .
\]
As $ (c,c) = k$, \eqref{pyPy2} is valid for  $y=c$. Now let $y$ in $\Omega^{(k)}$. There exists an idempotent $c$ of rank $k$ and some $\ell \in L$ such that $y=\ell c$. Now $p_y = \ell^*p_c\ell$ and $P(y) = \ell^*P(c)\ell$ and \eqref{pyPy2} follows by using the result obtained for $c$ when applied to $w=\ell v$.
\end{proof}

\begin{theorem} Let $3\leq k\leq r-1$. Let 
\[r_k(y) = \frac{4}{k(k-1)(kd+2)d}\, (kP(y)-p_y)\ .
\]
Then 
\begin{equation}\label{rk}
\int_{\Omega^{(k)}} e^{-\tr y} r_k(y) d\nu_k(y) = \Id_\mathbb J
\end{equation}
\end{theorem}
\begin{proof} Rewrite Theorem \ref{rmuOmega} using the Riesz integrals (cf \eqref{Riesz}), extended to operator-valued functions  to obtain
\begin{equation}\label{Rieszrmu}
\Id_\mathbb J = \frac{1}{\mu(\mu+1)\big(\mu-\frac{d}{2}\big)} \left(T_\mu(y)\,, \,\mu P(y)-\frac{d}{2}\,p_y\right) \ .
\end{equation}
The statement is valid for $\Re(\mu)$ large enough, and both sides can be continued analytically. By analytic continuation, the two sides coincide at $\mu = k\frac{d}{2}$ for $k=2,\dots, r-1$, and, as $d\nu_k= T_{k\frac{d}{2}}$, \eqref{rk} is just an equivalent formulation. 
 \end{proof}
 The positivity of $r_k(y)$ for $y\in \mathcal O^{(1)}$ follows from Lemma \ref{lemmapyPy2} and hence the values $\displaystyle \mu=k\frac{d}{2},\ 2\leq k\leq  r-1$ belong to the Wallach set.
 
 \section{The case $\mu=\frac{d}{2}$}
The last case concerns the case where $\mu= \frac{d}{2}$. In fact, the factor $(\mu-\frac{d}{2})$ in the denominator of the right hand expression  of \eqref{Rieszrmu} leads to an indetermination, as for $c$ a primitive (i.e. $\rank (c) = 1)$ idempotent,
\[p_c(v) = P(c)v=(v,c)\, c\ ,
\]
and hence the expression $\mu P(y)-\frac{d}{2} p_y$ vanishes for $\mu=\frac{d}{2} $ identically  on $ \Omega^{(1)}$. 
\begin{proposition}\label{no}
 There is no $\Herm^+(\mathbb J)$-valued measure supported on $\overline {\Omega^{(1)}}$ such that conditions \eqref{invR'} and \eqref{normR'}
are satisfied.
\end{proposition}

Before giving the proof, a few elementary lemmas are needed.
\begin{lemma}\label{ells}
Let $c$ be an idempotent in $J$. For $s>0$, let
$\ell_s= P(c+s(e-c))$. Then 
\[ \ell_s= \id \text{ on } \mathbb J(c,1),\qquad
 \ell_s = s \id  \text{ on } \mathbb J(c,\frac{1}{2})\qquad
 \ell_s = s^2 \id  \text{ on } \mathbb J(c,0)\ .\]
\end{lemma}
\begin{proof} By using the definition of the quadratic operator $P$
\[P\big(s+(e-c)\big) = P(c)+2s\big((L(c)L(e-c)+L(e-c)L(c)\big) + s^2 L(e-c)^2
\]
and the verification of the lemma is a routine calculation.
\end{proof}
\begin{lemma}\label{Sc}
 Let $c$ be a primitive idempotent of $J$. Let $L^c$ be the stabilizer of $c$ in $L$. Let $S\in \Herm^+(\mathbb J)$ and assume that
\begin{equation}\label{covS}
\forall \ell\in L^c,\qquad S=\ell \,S\,\ell^t\ .
\end{equation}
Then there exists a scalar $\alpha\in [0, +\infty)$ such that $S=\alpha P(c)$.
\end{lemma}
\begin{proof} Recall first the formula $P(\ell x) = \ell P(x)\ell^t$ for any element $x$ and $\ell\in L$. Hence $P(c)$ satisfies \eqref{covS}. Now assume that $S\in \Herm^+(\mathbb J)$ satisfies \eqref{covS}. For $s>0$,  $\ell_s= P\big(c+s(e-c)\big)$  belongs to  $L^{c}$ and is symmetric. Hence condition \eqref{covd/2} implies that $S= \ell_s\,S\,\ell_s$ for any $s>0$.
  Let $v\in \mathbb J(c,\frac{1}{2})$. Then by Lemma \ref{ells}
\[\ell_s \,S\,v= \frac{1}{s}\,S \,v\]
which is clearly impossible unless $Sv=0$. Hence $S$ vanishes on $J(c,\frac{1}{2})$. A similar argument shows that  $S$ vanishes on $J(c,0)$. As $c$ is primitive $J(c,1)=\mathbb Rc_1$ and $Sc_1 = \ell_sS\ell_s c_1= \ell_s Sc_1$, which forces $Sc_1\in \mathbb C c_1$. The conclusion follows.
\end{proof}

Now we are in condition to prove Proposition \ref{no}. Assume there exists a measure, which we denote by  $dR_1$ for simplicity, satisfying both conditions. Then there would exist  a $\Herm^+(\mathbb J)$-valued function $r_1$ on $\Omega^{(1)}$ such that $dR_1(x) = r_1(x)\, d\nu_1(x)$
satisfying, for $x\in \Omega^{(1)}$
\begin{equation}\label{covd/2}
r_1(\ell x) =  \ell\, r_1(x)\,\ell^t
\end{equation}
and 
\begin{equation}\label{normd/2}
\int_{{\overline \Omega}_1}e^{-\tr x} r_1(x)\,  d\nu_k(x)= \Id_{\mathbb J}
\end{equation}
Let $c$ be a primitive idempotent and let $S=r_1(c)$. Then $S$ satisfies the conditions of Lemma \ref{Sc} and hence there exists $\alpha\in [0,+\infty)$ such that $S= \alpha P(c)$. Now all primitive idempotents are conjugate by some element of $K$ and $S$ commutes with $K$ (a consequence of \eqref{invR'}).  Hence the constant $\alpha$ does not depend on $c$. Any element $x$ of $\Omega^{(1)}$ is a multiple of a primitive idempotent and hence satisfies $x=\vert x\vert c$ for some primitive idempotent $c$. In turn, this implies that \[r_1(x) = \vert x\vert^2 r_1(c) =\alpha \vert x\vert^2 P(c) = \alpha P(x)\ .\]
Hence
\[\Tr r_1(x)= \alpha \vert x\vert^2, \qquad (r_1(x) e\vert  e ) = \alpha (x^2\vert e) = \alpha \vert x\vert^2
\]
As $\Tr r_1(x)=(r_1(x)e\vert e)$ for any $x\in \Omega^{(1)}$, the integral on the left hand side of \eqref{normd/2}, call it $\Sigma$, satisfies $\Tr(\Sigma) = (\Sigma e,e)$. As $\Tr \Id_\mathbb J= n$ and $(\Id e \vert e) = r$, $\Sigma$ cannot be equal to $\Id_\mathbb J$. This achieves the proof of Proposition \ref{no}.
\smallskip

The proof of Theorem \ref{Wallachset} is now complete.

\footnotesize{\noindent Address\\ Jean-Louis Clerc, Universit\'e de Lorraine, CNRS, IECL, F-54000 Nancy, France
\medskip

\noindent \texttt{{jean-louis.clerc@univ-lorraine.fr
}}

\end{document}